\documentclass[12pt,amscd]{amsart}
\usepackage[all]{xy}
\usepackage{graphicx, tikz}
\usepackage{amsmath,amsxtra,amssymb,latexsym, amscd,amsthm}
\usepackage{indentfirst}
\usepackage[mathscr]{eucal}
\usepackage[pagebackref=true]{hyperref}

\newtheorem{thm}{Theorem}[section]
\newtheorem{cor}[thm]{Corollary}
\newtheorem{lem}[thm]{Lemma}
\newtheorem{prop}[thm]{Proposition}
\theoremstyle{definition}
\newtheorem{defn}{Definition}
\newtheorem{exm}[thm]{Example}

\newtheorem{rem}[thm]{Remark}

\newtheorem*{conj}{Conjecture}
\DeclareMathOperator{\NN}{\mathbb {N}}
\DeclareMathOperator{\ZZ}{\mathbb {Z}}

\DeclareMathOperator{\depth}{depth}

\DeclareMathOperator{\supp}{supp}

\def\H {\mathcal H}

\def\a {\mathbf a}
\def\b {\mathbf b}

\def\w {\mathbf w}

\def\m {\mathfrak m}

\begin{document}

\title []{The sequentially Cohen-Macaulay property of edge ideals of edge-weighted graphs  }

\author{Ly Thi Kieu Diem}
\address{Department of Mathematics and Informatics, Hanoi National University of Education, 136 Xuan Thuy, Hanoi, Vietnam}
\email{lykieudiem2001@gmail.com}

\textbf{\author{Nguy\^en C\^ong Minh}
\address{School of Applied Mathematics and Informatics\\ Hanoi University of Science and Technology, 1 Dai Co Viet, Hanoi, Vietnam}
\email{minh.nguyencong@hust.edu.vn}}

\author{Thanh Vu}
\address{Hanoi University of Science and Technology, 1 Dai Co Viet, Hanoi, Vietnam}
\email{vuqthanh@gmail.com}


\date{}

\dedicatory{Dedicated to Professor Ngo Viet Trung on the occasion of his 70th birthday}
\commby{}
%-----------------------------------------------------------
% -----------------------------------------------------------
\maketitle
% -----------------------------------------------------------
\begin{abstract}
    Let $I(G_{\w})$ be the edge ideal of an edge-weighted graph $(G,\w)$. We prove that $I(G_\w)$ is sequentially Cohen-Macaulay for all weight functions $\w$ if and only if $G$ is a Woodroofe graph.
\end{abstract}

\maketitle

\section{Introduction}
\label{sect_intro}

Let $S = K[x_1,\ldots,x_n]$ be a standard graded polynomial ring over an arbitrary field $K$. Let $G$ be a simple graph with vertex set $V=\{x_1,\ldots,x_n\}$ and edge set $E(G)$. By abuse of notation, we also use $x_ix_j$ to denote an edge $\{x_i,x_j\}$ of $G$. Assume that $\w:E(G)\rightarrow \ZZ_{>0}$ is a weight function on edges of $G$. The edge ideal of the edge-weighted graph $(G,\w)$ is defined by
$$I(G_\w) = \big( (x_i x_j)^{\w(x_ix_j)} \mid \{i,j\}\in E(G)\big) \subseteq S.$$
In particular, if every edge of $G$ has weight one then $I(G_\w)$ becomes the usual edge ideal $I(G)$. 

Edge ideals of edge-weighted graphs were introduced by Paulsen and Sather-Wagstaff \cite{PS}. In this work, the authors described a primary decomposition of $I(G_\w)$ and studied the Cohen-Macaulay property of $I(G_\w)$ when the underlying graph $G$ is a cycle, a tree, or a complete graph. In particular, they proved that $I(G_\w)$ is Cohen-Macaulay for all weight functions $\w$ when $G$ is a complete graph. In our first main result, we prove the converse of this result.

\begin{thm}\label{thm_CM} Let $G$ be a simple graph. The following statements are equivalent: 
	\begin{enumerate}
		\item  $I(G_\w)$ is Cohen-Macaulay for all weight functions $\w$; 
		\item  $I(G_\w)$ is Cohen-Macaulay for all weight functions $\w$ such that $\w(x_ix_j)\in\{1,2\}$ for all edges $x_i x_j \in E(G)$; 
		\item  $G$ is a disjoint union of finitely many complete graphs.
	\end{enumerate}
\end{thm}
In \cite{FSTY}, Fakhari, Shibata, Terai, and S. Yassemi characterized the unmixed property of $I(G_\w)$ when $G$ is a very well-covered graph and proved that this is equivalent to the Cohen-Macaulay property of $I(G_\w)$. In this context, Terai \cite{T} proposed the following conjecture

\begin{conj}[Terai] Let $G$ be a Cohen-Macaulay very well-covered graph. Then $I(G_\w)$ is sequentially Cohen-Macaulay for all weight functions $\w$. 
\end{conj}
We first recall the definition of sequentially Cohen-Macaulay modules over $S$.
\begin{defn} Let $M$ be a graded module over $S$. We say that $M$ is
	sequentially Cohen-Macaulay if there exists a filtration
	$$0 = M_0\subset M_1 \subset \cdots \subset M_r = M$$
	of $M$ by graded $S$-modules such that $\dim (M_i/M_{i-1}) <\dim (M_{i+1}/M_i)$
	for all $i$, where $\dim$ denotes Krull dimension, and $M_i/M_{i-1}$ is Cohen-Macaulay for all $i$. An ideal $J$ is said to be sequentially Cohen-Macaulay if $S/J$ is a sequentially Cohen-Macaulay $S$-module. A graph $G$ (resp. $(G,\w)$) is said to be sequentially Cohen-Macaulay if $I(G)$ (resp. $I(G_\w)$) is. 
\end{defn}

The notion of sequentially Cohen-Macaulay was introduced by Stanley \cite{S} as a generalization of the Cohen-Macaulay property in connection with the work of Bj\"orner and Wachs on nonpure shellability \cite{BW1, BW2}. When $J$ is a sequentially Cohen-Macaulay ideal, it is well-known that $J$ is Cohen-Macaulay if and only if $J$ is unmixed.

In motivation to study the conjecture of Terai, we can classify graphs for which $G_\w$ are sequentially Cohen-Macaulay for all weight functions $\w$. To introduce our result, we first define a special class of simple graphs which contain $5$-cycles and chordal graphs. A chordless cycle $C_t$ of length $t$ is a cycle with no chord $\{i,j\}$ for $j\ne i+1$. Equivalently, the induced graph of $G$ on $\{1,\ldots, t\}$ is the cycle on $t$ vertices.

\begin{defn}  A simple graph $G$ is said to be a Woodroofe graph if $G$ has no chordless cycles of length other than $3$ or $5$. 
\end{defn}

In \cite[Theorem 1]{Wo}, Woodroofe proved that if $G$ is a Woodroofe graph, then it is vertex-decomposable. So, it is sequentially Cohen-Macaulay. Our second main result of this paper states that Woodroofe graphs are precisely graphs for which $G_\w$ are sequentially Cohen-Macaulay for all weight functions $\w$. 

\begin{thm}\label{thm_sCM} Let $G$ be a simple graph. The following statements are equivalent: 
	\begin{enumerate}
		\item  $(G,\w)$ is sequentially Cohen-Macaulay for all weight functions $\w$; 
		\item $(G,\w)$ is sequentially Cohen-Macaulay for all weight functions $\w$ such that $\w(x_ix_j)\in\{1,2\}$ for all edges $x_i x_j \in E(G)$; 
		\item  $G$ is a Woodroofe graph. 
	\end{enumerate}
\end{thm}

Now we explain the organization of the paper. In Section 2, we prove Theorem \ref{thm_CM} and Theorem \ref{thm_sCM}. In Section 3, we give some applications of our main results. In particular, we provide counterexamples to Terai's conjecture.

\section{Proof of the main results}
Throughout the paper, we denote $S = K[x_1,\ldots, x_n]$ a standard graded polynomial ring over a field $K$. Let $\m = (x_1,\ldots, x_n)$ be the maximal homogeneous ideal of $S$. We first recall some notation and results.

For a finitely generated graded $S$-module $L$, the depth of $L$ is defined to be
$$\depth(L) = \min\{i \mid H_{\m}^i(L) \ne 0\},$$
where $H^{i}_{\m}(L)$ denotes the $i$-th local cohomology module of $L$ with respect to $\m$. 

\begin{defn}
    A finitely generated graded $S$-module $L$ is called Cohen-Macaulay if $\depth L = \dim L$. A homogeneous ideal $I\subseteq S$ is said to be Cohen-Macaulay if $S/I$ is a Cohen-Macaulay $S$-module.
\end{defn}

Let $I$ be a monomial ideal in $S$. In \cite{H}, Hochster introduced the set of associated radical ideals of $I$, namely $\sqrt{I:u}$ for monomials $u\notin I$ and proved that the Cohen-Macaulay property of $I$ can be characterized in terms of its associated radicals. In \cite{JS}, Jafari and Sabzrou showed that the sequentially Cohen-Macaulay property of $I$ can also be characterized in terms of its associated radicals. 

\begin{lem}\label{lem_radical_transfer}
    A monomial ideal $I$ is (sequentially) Cohen-Macaulay if and only if $\sqrt{I:u}$ is (sequentially) Cohen-Macaulay for all monomials $u \notin I$.
\end{lem}
\begin{proof} The statement for Cohen-Macaulay property follows from \cite[Theorem 7.1]{H} and the fact that $\dim S/\sqrt{I:u} \le \dim S/I$ for all monomials $u \notin I$. See also \cite[Proposition 2.8]{JS}. The statement for sequentially Cohen-Macaulay is \cite[Proposition 2.23]{JS}.
\end{proof}

The associated radicals also play an important role in studying the regularity of monomial ideals (see \cite{MNPTV}). The associated radicals can be computed as follows. For a monomial $f$ in $S$, the support of $f$, denoted $\supp(f)$, is the set of $x_i$ such that $x_i|f$. The radical of $f$ is defined by $\sqrt{f} = \prod_{x_i \in \supp f} x_i$. For an exponent $\a \in \NN^n$, denote $x^\a$ the monomial $x_1^{a_1} \cdots x_n^{a_n}$.

\begin{lem}\label{radical_colon} Let $J$ be a monomial ideal in $S$ generated by the monomials $f_1, \ldots, f_r$ and $x^\a$ a monomial in $S$. Then $\sqrt{I:x^\a}$ is generated by $$\sqrt{f_1/\gcd(f_1, x^\a)}, \ldots, \sqrt{f_r/\gcd(f_r,x^\a)}.$$
\end{lem}
\begin{proof}
	See \cite[Lemma 2.24]{MNPTV}.
\end{proof}

Let $G$ denote a finite simple graph over the vertex set $V(G)= \{x_1,x_2,\ldots,x_n\}$ and the edge set $E(G)$. A subgraph $H=G[W]$ is called an induced subgraph of $G$ on $W\subset V(G)$ if for any vertices $u,v\in W$ then $uv\in E(H)$ if and only if $uv\in E(G)$. Let $\w: E(G) \to \ZZ_{>0}$ be a weight function on the edges of $G$. We first show that the property that $I(G_\w)$ are (sequentially) Cohen-Macaulay for all weight functions $\w$ is equivalent to the property that all induced subgraphs of $G$ are (sequentially) Cohen-Macaulay.

\begin{lem}\label{lem_induced_subgraphs} Let $G$ be a simple graph. The following statements are equivalent. 
\begin{enumerate}
		\item $(G,\w)$ is (sequentially) Cohen-Macaulay for all weight functions $\w$; 
  \item $(G,\w)$ is (sequentially) Cohen-Macaulay for all weight functions $\w$ such that $\w(x_ix_j)\in\{1,2\}$ for all edges $x_i x_j \in E(G)$; 
		\item $G[W]$ is (sequentially) Cohen-Macaulay for all subsets $W \subseteq V(G)$. 
	\end{enumerate}
\end{lem}
\begin{proof} It is obvious that $(1) \Rightarrow (2)$. Now, we prove    $(2) \Rightarrow (3)$. Let $W$ be any subset of $V(G)$. Let $\w$ be the weight function defined as follows: 
    $$\w(e)=
\begin{cases}
2\text{ if } e \in G[W],\\
1\text{ otherwise}.
\end{cases}$$
Let $x^\a = \prod_{x_j \in W} x_j$. We first show that 
\begin{equation}\label{eq_colon}
    \sqrt{I(G_\w):x^\a} = I(G[W]) + (\text{some variables not in W}) + I(G[W']),
\end{equation}
where $W'$ is a subset of $V(G) \setminus W$. By Lemma \ref{radical_colon} $\sqrt{I(G_\w): x^\a}$ is generated by $\sqrt{e^{\w(e)}/\gcd(e^{\w(e)},x^\a)}$ for all edge $e$ of $G$. We have three cases: 

\vspace{1mm}
\noindent \textbf{Case 1.} $e \in G[W]$. In this case, $\w(e) = 2$ and $e^2/\gcd(e^2,x^\a) = e$.

\vspace{1mm}
\noindent \textbf{Case 2.} $|\supp e \cap W| = 1$. Assume that $e = xy$ with $x \in W$ and $y\notin W$. Then $e/\gcd(e,x^\a) = y \notin W$. 

\vspace{1mm}
\noindent \textbf{Case 3.} $|\supp e \cap W| = 0$. Then $e / \gcd(e,x^\a) = e$. 

Eq. \eqref{eq_colon} follows. By Lemma \ref{lem_radical_transfer}, $\sqrt{I(G_\w):x^\a}$ is (sequentially) Cohen-Macaulay. In particular, $I(G[W]) + I( G[W'])$ is (sequentially) Cohen-Macaulay. Since $W \cap W' = \emptyset$, by \cite[Lemma 4.1]{V} and \cite[Lemma 20]{Wo}, we deduce that $I(G[W])$ is (sequentially) Cohen-Macaulay.

$(3) \Rightarrow (1).$ By Lemma \ref{radical_colon}, any minimal generator of $\sqrt{I(G_\w) : x^\a}$ for any weight function $\w$ and any monomial $x^\a$ such that $x^\a \notin I(G_\w)$ is either $xy$ where $xy$ is an edge of $G$ or a variable. Hence, $\sqrt{I(G_\w):x^\a} = I(G[W]) + (\text{some variables})$ for some subset $W$ of $V(G)$. By assumption, they are (sequentially) Cohen-Macaulay. By Lemma \ref{lem_radical_transfer}, $I(G_\w)$ is (sequentially) Cohen-Macaulay.
\end{proof}

By Lemma \ref{lem_induced_subgraphs}, we see that studying the class of graphs for which $I(G_\w)$ is (sequentially) Cohen-Macaulay is equivalent to the problem of finding obstructions to (nonpure) shellability of flag complexes. This observation leads us to the proof of the main results.

\begin{proof}[Proof of Theorem \ref{thm_CM}] By Lemma \ref{lem_induced_subgraphs}, the Theorem follows from the following facts.
\begin{enumerate}
    \item Disjoint unions of complete graphs are Cohen-Macaulay \cite{V}.
    \item Induced subgraphs of a disjoint union of complete graphs are disjoint unions of complete graphs.
    \item $P_3$ a path of length $2$ is not Cohen-Macaulay.
    \item $P_3$ is not an induced subgraph of $G$ then $G$ is a disjoint union of complete graphs.
\end{enumerate}
\end{proof}

\begin{proof}[Proof of Theorem \ref{thm_sCM}] By Lemma \ref{lem_induced_subgraphs}, the Theorem follows from the definition of Woodroofe graphs and the following facts.
\begin{enumerate}
    \item Woodroofe graphs are sequentially Cohen-Macaulay \cite[Theorem 1]{Wo}.
    \item Induced subgraphs of a Woodroofe graph are Woodroofe graphs.
    \item The cycles $C_t$ are not sequentially Cohen-Macaulay for $t \neq 3,5$ (see \cite[Proposition 4.1]{FT} and \cite[Theorem 10]{Wo}).
\end{enumerate}
\end{proof}

\begin{rem} 
\begin{enumerate}
    \item The simplicial complexes corresponding to edge ideals of disjoint unions of complete graphs are matroid of rank $2$. 
    \item One can show that the classification of hypergraphs $\H$ whose edge ideals of edge-weighted hypergraphs $(\H,\w)$, $I(\H_\w)$ is (sequentially) Cohen-Macaulay for all weight functions $\w$ also reduces to study the obstructions to (nonpure) shellablity of simplicial complexes. 
    \item We can show that $I(\H_\w)$ is Cohen-Macaulay for all weight functions $\w$ if and only if $\Delta(I_\H)$ is a matroid. 
    \item The obstruction to nonpure shellability is much more subtle (see \cite{Wa} for some partial results) and we will leave that for future work.
\end{enumerate}
\end{rem}

\section{Applications}

In this section, we will give some applications of our results. Firstly, one knows that when $J$ is sequentially Cohen-Macaulay, then $J$ is Cohen-Macaulay if and only if it is unmixed. Therefore, we obtain

\begin{cor} Let $G$ be a Woodroofe graph and $\w: E(G) \to \ZZ_{>0}$ a weight function. Then $I(G_{\w})$ is Cohen-Macaulay if and only if $I(G_{\w})$ is unmixed.
\end{cor}

When $G$ is a Cohen-Macaulay graph, a weight function $\w$ on edges of $G$ is called Cohen-Macaulay if $(G,\w)$ is Cohen-Macaulay. Before giving our next application, we recall the result of Paulsen and Sather-Wagstaff \cite[Theorem 4.4]{PS} on edge-weighted graph $(C_5,\w)$. They proved that $\w$ is Cohen-Macaulay if and only if there exists a vertex $v$ so that the weights on edges of $C_5$ starting from $v$ in clockwise order are of the form $a,b,c,d,a$ and that $a \le b \ge c\le d \ge a$. We call such a vertex $v$ a balancing vertex of $\w$. 

Let $H$ be a graph formed by connecting two $5$-cycles by a path. By \cite[Theorem 2.4]{HMT}, $H$ is Cohen-Macaulay if and only if this path is of length $1$. We may assume that the vertices of $H$ are $\{x_1, \ldots, x_5, y_1, \ldots, y_5\}$ and edges of $H$ are $\{x_1x_2, \ldots, x_4x_5,x_1x_5,y_1y_2, \ldots,y_4y_5,y_1y_5,x_1y_1\}$. Note that $I(H) + (x_i)$ and $I(H) + (y_i)$ are not Cohen-Macaulay for $i \in \{2,5\}$. With this assumption, we have
\begin{prop} The edge-weighted graph $(H,\w)$ is Cohen-Macaulay if and only if $\w$ satisfies the following conditions:
\begin{enumerate}
    \item $\w(x_1y_1) \le \min \{\w(x_1x_2),\w(x_1x_5),\w(y_1y_2),\w(y_1y_5)\},$
    \item The induced edge-weighted graphs of $(H,\w)$ on $\{x_1,\ldots,x_5\}$ and $\{y_1,\ldots,y_5\}$ are Cohen-Macaulay.
    \item Balancing vertices of $\w$ on $\{x_1, \ldots,x_5\}$ and 
    $\{y_1,\ldots,y_5\}$ can be chosen among $\{x_1,x_3,x_4\}$ and $\{y_1,y_3,y_4\}$ respectively.
\end{enumerate}
\end{prop}
\begin{proof} Denote $I = I(H_\w)$. Let $(H_1,\w_1)$ and $(H_2,\w_2)$ be the induced edge-weighted graphs of $(H,\w)$ on $\{x_1, \ldots, x_5\}$ and $\{y_1, \ldots, y_5\}$ respectively. 

First, we prove that if $\w$ is Cohen-Macaulay, it must satisfy the above conditions. For (1), assume by contradiction that $\w(x_1y_1) = a > \w(y_1y_2) = b$. Let $c = \max(\w(y_3y_4),\w(y_4y_5))$. Then 
$$\sqrt{I : y_1^{a-1} y_4^c} = I (H[x_1,\ldots,x_5,y_1]) + (y_2,y_3,y_5).$$
In particular, it is not Cohen-Macaulay. By Lemma \ref{lem_radical_transfer}, $I(H_\w)$ is not Cohen-Macaulay, a contradiction. By symmetry, $\w$ must satisfy condition (1).

We now prove that $(H_2, \w_2)$ must be Cohen-Macaulay. Assume by contradiction that, $(H_2,\w_2)$ is not Cohen-Macaulay. By Lemma \ref{lem_radical_transfer}, there exists an exponent $y^\b$ such that $\sqrt{I(H_2,\w_2) : y^\b}$ is not Cohen-Macaulay. Then we have 
$$\sqrt{I(H_\w) : x_2^{a_2}x_4^{a_4} y^\b} = \sqrt{I(H_2,\w_2):y^\b} + (x_1,x_3,x_5),$$
where $a_2 = \max(\w(x_2x_1),\w(x_3x_2))$ and $a_4 = \max(\w(x_3x_4),\w(x_4x_5))$. In particular, it is not Cohen-Macaulay. By Lemma \ref{lem_radical_transfer}, $I(H_\w)$ is not Cohen-Macaulay, a contradiction. By symmetry, $\w$ must satisfy condition (2).

Now note that if $\w(x_2x_3) < \w(x_3x_4)$ then $\sqrt{I:x_3^b} = I + (x_2)$ where $b = \w(x_3x_4) - 1$. Since $I + (x_2)$ is not Cohen-Macaulay, this implies a contradiction. Hence, $\w(x_2x_3) \ge \w(x_3x_4)$. By symmetry, we deduce that $\w(x_4x_5) \ge \w(x_3x_4)$. By \cite[Theorem 4.4]{PS} and the previous claim that $(H_1, \w_1)$ is Cohen-Macaulay, we deduce that a balancing vertex of $\w$ on $\{x_1, \ldots, x_5\}$ can be chosen among $\{x_1,x_3,x_4\}$. By symmetry, $\w$ must satisfy condition (3). 

It remains to prove that if $\w$ satisfies conditions (1), (2), (3), then $I = I(H_\w)$ is Cohen-Macaulay. By Lemma \ref{lem_radical_transfer}, it suffices to prove that $\sqrt{I:x^\a y^\b}$ is Cohen-Macaulay for all exponents $\a, \b$ such that $x^\a y^\b \notin I$. Denote $I_{\a,\b} = \sqrt{I:x^\a y^\b}$. First, we have

\vspace{1mm}
\noindent \textbf{Claim A.} Assume that $(G,\w)$ is an edge-weighted graph and $x^\a$ is an exponent such that $a_i < \w(e)$ for all edges $e$ adjacent to $i$ then 
\begin{equation}
    \sqrt{I(G_\w):x^\a} = \sqrt{I(G_\w):x^\b},
\end{equation}
with $x^\b = x_1^{a_1} \cdots x_{i-1}^{a_{i-1}} x_{i+1}^{a_{i+1}} \cdots x_n^{a_n}$. In other words, we may assume that $a_i = 0$.

By symmetry, we may assume that $a_1 \ge b_1$. Since $x^\a y^\b \notin I(H_\w)$, we must have $b_1 < \w(x_1y_1) \le \min (\w(y_1y_2), \w(y_1y_5))$. By Claim A, we may assume that $b_1 = 0$. There are two cases as follows.

\vspace{1mm}
\noindent \textbf{Case 1.} $a_1 \ge \w(x_1y_1)$. Then

\begin{equation}
    I_{\a,\b} = \sqrt{I : x^\a y^\b} = (y_1) + \sqrt{I(H_1,\w_1):x^\a} + \sqrt{I(H_2,\w_2):y^\b}
\end{equation}
Since $(H_1,\w_1)$ is Cohen-Macaulay by \cite[Theorem 4.4]{PS}, $I_{\a,\b}$ is not Cohen-Macaulay if and only if $(y_1) + \sqrt{I(H_2,\w_2):y^\b} = (y_1,y_2) + I(H_2)$ or $I(H_2) + (y_1,y_5)$. Assume by contradiction that $(y_1) + \sqrt {I(H_2,\w_2):y^\b} = (y_1,y_2) + I(H_2)$. Then we must have $b_3 \ge \w(y_2y_3) \ge \w(y_3y_4)$. But then $y_4 \in \sqrt{I(H_2,\w_2):y^\b}$, a contradiction. Hence, $I_{\a,\b}$ is Cohen-Macaulay.

\vspace{1mm}
\noindent \textbf{Case 2.} $a_1 < \w(x_1y_1)$. By Claim A, we may assume that $a_1 = 0$. In particular, 
\begin{equation}
    I_{\a,\b} = \sqrt{I(H_1,\w_1):x^\a} + \sqrt{I(H_2,\w_2):y^\b} + (x_1y_1).
\end{equation}
If $x_1$ or $y_1$ appears in $I_{\a,\b}$, with an argument similar to Case 1, we deduce that $I_{\a,\b}$ is the sum of two Cohen-Macaulay ideals on different sets of variables and an ideal generated by some other variables. Hence, it is Cohen-Macaulay. Thus, we may assume that $x_1,y_1$ does not appear in $I_{\a,\b}$. By Lemma \ref{radical_colon}, $\sqrt{I(H_1,\w_1):x^\a} = I(H_1) + (x_i\mid i \in W_1)$, where $W_1\subseteq \{1,\ldots, 5\}$. By the following facts, $W_1$ must belong to $P = \{ \{2,4\}, \{2,3,4\}, \{3,5\},\{3,4,5\}, \{3\}, \{4\}, \emptyset\}$.
\begin{enumerate}
    \item By assumption, $1 \notin W_1$.
    \item Since the balancing vertex of $(H_1,\w_1)$ can be chosen in the set $\{3,4,1\}$, $\w(x_2x_3) \ge \w(x_3x_4)$. Hence, if $2 \in W_1$ then $4 \in W_1$. Similarly, if $5 \in W_1$ then $3 \in W_1$.
    \item By \cite[Theorem 4.4]{PS} and Lemma \ref{lem_radical_transfer}, $I(H_1) + (x_i\mid i\in W_1)$ is a Cohen-Macaulay ideal, $W_1$ cannot be $\{3,4\}$. 
    \item $W_1$ cannot contain $\{2,5\}$.
\end{enumerate}
Proof of (4). Assume by contradiction that $2,5 \in W_1$. Since $a_1 = 0$, we must have $a_3 \ge \w(x_2x_3)$. Similarly, $a_4 \ge \w(x_4x_5)$. Since the balancing vertex of $(H_1,\w_1)$ can be chosen in the set $\{3,4,1\}$, we have $\w(x_3x_4) \le \min (\w(x_2x_3),\w(x_4x_5))$. Hence, $\w(x_3x_4) \le a_3,a_4$. But this implies that $x^\a \in I$, a contradiction.

Now, it is easy to check that if $W_1, W_2$ belong to $P$, we have $I(H) + (x_i \mid i \in W_1) + (y_j \mid j \in W_2)$ is Cohen-Macaulay. The Proposition follows.
\end{proof}

Finally, by Theorem \ref{thm_sCM}, any Cohen-Macaulay very well-covered graph that is not Woodroofe is a counterexample to Terai's conjecture. We provide some concrete examples below. Recall that a simple graph is called very well-covered if the size of every minimal vertex cover is half the number of vertices. In particular, it is unmixed. 

\begin{exm}\label{counterexm-Terai's} 
Let $H$ be a suspension of a cycle $C_t$ for $t\ne 3,5$; i.e., the set of edges and the set of vertices are
$$E(H)=\{x_1x_2, x_2x_3,\ldots,x_{t-1}x_t,x_tx_1, x_1y_1,\ldots,x_ty_t\} \text{ and } V(H)=\{x_1,y_1,\ldots,x_t,y_t\}.$$
Let $\w$ be a weight function on $E(H)$ taking value $w \ge 2$ for the edges $x_{i}x_{i+1}$ and value $1$ otherwise. Then, $H$ is a Cohen-Macaulay very well-covered graph, but $(H,\w)$ is not sequentially Cohen-Macaulay.
\end{exm}
\begin{proof} The graph $H$ is Cohen-Macaulay by  \cite[Theorem 2.1]{SVV} (also see \cite{V}). By definition, $H$ is very well-covered. Since
	$$\sqrt{I(H_{\w}):\prod_{i=1}^t x_i^{w-1}}=I(C_t)+(y_1,\ldots,y_t)$$
 and $I(C_t)$ is not sequentially Cohen-Macaulay by \cite[Proposition 4.1]{FT}. By Lemma \ref{lem_radical_transfer}, $I(H_\w)$ is not sequentially Cohen-Macaulay.
\end{proof}

\subsection*{Acknowledgments} This paper was done while the first author was visiting the Vietnam Institute for Advanced Study in Mathematics (VIASM). He would like to thank the VIASM for the hospitality and financial support, and he also thanks the Vietnam National Foundation for Science and Technology Development (NAFOSTED) for its support under grant number 101.04-2021.19.

\subsection*{Conflict of interest}

The authors declare no potential conflict of interests.

\end{document}